\documentclass{amsart}

\usepackage{amsmath}
\usepackage{amssymb, amscd, fixmath}
\usepackage{amsthm}
\usepackage{enumerate}
\usepackage[cmtip,all]{xy}

\theoremstyle{plain}
\newtheorem{thm}{Theorem}[section]
\newtheorem{lem}[thm]{Lemma}
\newtheorem{prop}[thm]{Proposition}

\theoremstyle{remark}

\numberwithin{equation}{section}

\def\End{\operatorname{End}}

\def\Hom{\operatorname{Hom}}
\def\Ind{\operatorname{Ind}}
\def\Irr{\operatorname{Irr}}
\def\Isom{\operatorname{Isom}}

\newcommand{\la}{\langle}
\newcommand{\ra}{\rangle}

\def\GL{\mathrm{GL}}

\def\Sp{\mathrm{Sp}}

\def\rH{\mathrm{H}}
\def\rD{\mathrm{D}}
\def\rF{\mathrm{F}}
\def\rG{\mathrm{G}}
\def\rP{\mathrm{P}}

\def\CC{\mathbb{C}}

\def\ZZ{\mathbb{Z}}

\newcommand{\abs}[1]{\lvert#1\rvert}

\begin{document}

\title[The Howe duality conjecture]
{The Howe duality conjecture:\\ Quaternionic case}
\author{Wee Teck Gan}
\address{Department of Mathematics, National University of Singapore, 10 Lower Kent Ridge Road, Singapore 119076}
\email{matgwt@nus.edu.sg}
\author{Binyong Sun}
\address{Hua Loo-Keng Key Laboratory of Mathematics, Institute of Mathematics,
Academy of Mathematics and Systems Science, CAS, Beijing, 100190, P.R. China}
\email{sun@math.ac.cn}

\subjclass[2000]{Primary 11F27, Secondary 22E50}

\keywords{Howe duality conjecture, theta correspondence, quaternionic dual pair}
\date{\today}

 \dedicatory{In celebration of  \\
Professor Roger Howe's  70th birthday}

\begin{abstract}
We complete the proof of the Howe duality conjecture in the theory of local theta
correspondence by treating the remaining case of  quaternionic dual pairs in arbitrary residual characteristic.
\end{abstract}

\maketitle

%{\em  Doing so we find that the duality conjecture is certainly
%true if $G$ or $G'$ is compact, and in general is Òalmost?true. It remains in the
%general case to remove the Òalmost".}

\section{\textbf{Introduction}}

Let $\rF$ be a non-archimedean local field of characteristic not $2$.
Let $W$ be a finite-dimensional symplectic vector space over $\rF$ with symplectic form $\langle \,,\, \rangle_W$.
Write
\begin{equation}\label{meta0}
   1\rightarrow \{1,\varepsilon_W\}\rightarrow
   \widetilde{\Sp}(W)\rightarrow
   \Sp(W)\rightarrow 1
\end{equation}
for the metaplectic double cover of the symplectic group $\Sp(W)$. It does not split unless $W=0$. Denote by
$\rH(W):=W\times \rF$ the Heisenberg group attached to $W$, with
group multiplication
\[
      (u, \alpha)(v,\beta):=(u+v, \alpha +\beta+\langle u,v\rangle_W), \qquad u,v\in W, \ \alpha,\beta\in \rF.
\]
Then $\widetilde{\Sp}(W)$ acts on $\rH(W)$ as group automorphisms
through the action of $\Sp(W)$ on $W$, and we may form the semi-direct
product $\widetilde{\mathrm J}(W):=\widetilde{\Sp}(W)\ltimes \mathrm
H(W)$.

 Fix an arbitrary non-trivial unitary
character $\psi: \rF\rightarrow \mathbb C^\times$. Up to isomorphism, there is a unique
smooth representation $\omega_\psi$ of $\widetilde{\mathrm J}(W)$ (called a Weil representation) such that (\emph{cf.} \cite[Section IV.43]{weil})
\begin{itemize}
  \item ${\omega_\psi}|_{\rH(W)}$ is irreducible and has central character $\psi$;
  \item $\varepsilon_W\in\widetilde{\Sp}(W)$ acts through the scalar multiplication by
  $-1$.
\end{itemize}
Unless $W=0$, the above second condition is a consequence of the first one.

Denote by $\tau$ the involution of $\End_\rF(W)$ specified by
\[
   \langle x\cdot u, v\rangle_W=\langle u, x^\tau\cdot v\rangle_W, \qquad u,v\in W,\,x\in \End_\rF(W).
\]
Let $(A, A')$ be a pair of $\tau$-stable semisimple $\rF$-subalgebras of $\End_\rF(W)$ which are mutual centralizers of each other. Put $G:=A\cap \Sp(W)$ and $G':=A'\cap \Sp(W)$, which are closed subgroups of $\Sp(W)$. Following Howe, we call the pair $(G,G')$ so obtained a {\em reductive dual pair} in $\Sp(W)$.
We say that the pair $(A, A')$ (or the reductive dual pair $(G,G')$) is irreducible of type I if  $A$ (or equivalently $A'$) is a simple algebra, and say that it is  irreducible of type II if  $A$ (or equivalently $A'$) is the product of two simple algebras which are exchanged by $\tau$. A complete classification of such dual pairs has been given by Howe.
\vskip 5pt

For every closed
subgroup $H$ of  $\Sp(W)$, write
$\widetilde H$ for the double cover of $H$ induces by the
metaplectic cover \eqref{meta0}. Then $\widetilde G$ and $\widetilde
G'$ commute with each other inside the group $\widetilde{\Sp}(W)$ (\emph{cf.} \cite[Chapter 2, Lemma II.5]{mvw}).
Thus, the Weil representation $\omega_{\psi}$ can be regarded as a representation of $\widetilde G \times \widetilde G'$.
\vskip 5pt

 For every $\pi\in \Irr(\widetilde G)$, put
 \[
  \Theta_{\psi}(\pi):=(\omega_\psi\otimes \pi^\vee)_{\widetilde G},
\]
 to be viewed as a smooth representation of $\widetilde G'$. Here and as usual, a superscript ``$\,^\vee$" indicates the contragredient representation, a subscript group indicates the coinvariant space, and  ``$\Irr$" indicates the set of isomorphism classes of irreducible admissible representations of the group. It was proved by Kudla \cite{k83} that the representation $\Theta_{\psi}(\pi)$ is admissible and has finite length. Denote by $\theta_{\psi}(\pi)$ the maximal semisimple quotient of $\Theta_{\psi}(\pi)$, which is called the theta lift of $\pi$. In this paper, we complete the proof of the following Howe duality conjecture.
\vskip 5pt
\noindent{\bf \underline{The Howe Duality Conjecture}}
%  \begin{thm} \label{T:howe}
\vskip 5pt

\noindent  For every reductive dual pair ($G, G')$ and  every $\pi\in \Irr(\widetilde G)$,  the theta lift $\theta_{\psi}(\pi)$ is irreducible if it is non-zero.
%\end{thm}
\vskip 5pt

 The Howe duality conjecture is easily reduced to the case when the pair $(A, A')$ is irreducible (of type I or II).
 It has been proved by Waldspurger \cite{w90} when the residual characteristic of $\rF$ is not $2$.
 For irreducible reductive dual pairs of type II, the conjecture was proved in full  and more simply by Minguez in \cite{mi}. Every  irreducible reductive dual pair of type I is an orthogonal-symplectic dual pair, a unitary dual pair, or a quaternionic dual pair \cite[Section 5]{H1}. For orthogonal-symplectic dual pairs and unitary dual pairs, the conjecture was proved in \cite{gt} (it was earlier proved in \cite{LST} that $\theta_{\psi}(\pi)$ is multiplicity free). For the remaining case of quaternionic dual pairs, only a partial result was obtained  in \cite{gt} (for Hermitian representations). The reason is that \cite{gt} makes use of the MVW-involution on the category of smooth representations, and it has  been shown in \cite{sun} that such an involution does not exist in the quaternionic case.
\vskip 5pt

The purpose of this paper is to explain how the use of the MVW-involution can be avoided, thus completing the proof of the Howe duality conjecture in the  quaternionic case.
The lack of an MVW-involution necessitates relating  the theta lifts of $\pi$ and $\pi^\vee$, and the key new ingredient is provided by
the following consequence of the conservation relations shown in \cite[Equalities (12)]{sz}.
\vskip 5pt

\begin{lem}\label{E:nonvan}
Assume that $(G,G')$ is irreducible. Then for every $\pi\in \Irr(\widetilde G)$,
\begin{equation}\label{E:nonvan0}
   \theta_\psi(\pi)\neq 0\qquad\textrm{if and only if}\qquad  \theta_{\bar \psi}(\pi^\vee)\neq 0,
\end{equation}
where  $\bar \psi$ denotes the complex conjugation of $\psi$.
\end{lem}
In proving the Howe duality conjecture, one needs to strengthen Lemma \ref{E:nonvan} to the identity
\begin{equation} \label{E:identity}
(\theta_{\psi}(\pi))^\vee\cong \theta_{\bar \psi}(\pi^\vee)   \quad \text{ for every $\pi\in \Irr(\widetilde G)$.} \end{equation}
  Hence, the main result of this paper  is the following theorem, which encompasses the Howe duality conjecture and the identity  (\ref{E:identity}).
\vskip 5pt

%as in the following theorem.
 %\begin{thm} \label{T:howe2}
% For every $\pi\in \Irr(\widetilde G)$, $(\theta_{\psi}(\pi))^\vee\cong \theta_{\bar \psi}(\pi^\vee)$.
%\end{thm}
% Similar to Theorem \ref{T:howe}, Theorem \ref{T:howe2} is reduced to the case when $(G,G')$ is irreducible. It is easy to see that the combination of Theorem \ref{T:howe} and Theorem \ref{T:howe2} is equivalent to the following proposition.

 \begin{thm} \label{T:howe3}
 Assume that $(G,G')$ is irreducible, and the size of $G$ is no smaller than that of $G'$. Then for all $\pi,\sigma\in \Irr(\widetilde G)$,
\begin{itemize}
      \item $\theta_\psi(\pi)$ is irreducible if it is non-zero;
      \item if $\theta_{\psi}(\pi)\cong \theta_\psi(\sigma)\neq 0$, then $\pi\cong \sigma$;
      \item $(\theta_\psi(\pi))^\vee\cong \theta_{\bar \psi}(\pi^\vee)$.
    \end{itemize}
 Consequently, the Howe duality conjecture holds for both $(G,G')$ and $(G',G)$, and for every $\pi'\in \Irr(\widetilde G')$, $(\theta_\psi(\pi'))^\vee\cong \theta_{\bar \psi}({\pi'}^\vee)$.
\end{thm}

Here the size of $G$ is defined to be
\[
   \mathrm{size}(G):=\left\{
                   \begin{array}{ll}
                    \frac{n_A}{2}+\frac{\dim_K A^{\tau=-1}}{n_A} , & \hbox{if $(G,G')$ is an orthogonal-symplectic or quaternionic dual pair;} \\
                     n_A, & \hbox{otherwise,}
                   \end{array}
                 \right.
\]
where $K$ denotes the center of $A$, $A^{\tau=-1}:=\{\alpha\in A\mid \alpha^\tau=-\alpha\}$, and $n_A$ denotes the integer such that $\mathrm{rank}_K A=n_A^2$.  The size of $G'$ is analogously defined.

%It turns out that the key new ingredient of the proof of Proposition \ref{T:howe3} is provided by
\vskip 5pt

In fact, exploiting Lemma \ref{E:nonvan}, Theorem \ref{T:howe3} is equivalent to the following proposition.

 \begin{prop} \label{T:mainq00}
 Assume that $(G,G')$ is irreducible, and the size of $G$ is no smaller than that of $G'$. Then for all $\pi,\sigma\in \Irr(\widetilde G)$,
 \begin{equation}\label{dimleq10}
   \dim \Hom_{\widetilde G'}(\theta_{\psi}(\pi)\otimes \theta_{\bar \psi}(\sigma),\mathbb C)\leq \dim \Hom_{\widetilde G}(\pi\otimes \sigma, \mathbb C).
 \end{equation}
\end{prop}

\vskip 5pt
In what follows, we show that Lemma \ref{E:nonvan} and Proposition \ref{T:mainq00}  imply Theorem \ref{T:howe3}.
\vskip 5pt

\noindent{\bf \underline{Proof of  (Lemma \ref{E:nonvan} $+$ Proposition \ref{T:mainq00} $\Longrightarrow$ Theorem \ref{T:howe3})}}
\vskip 5pt

 For every $\pi\in \Irr (\widetilde G)$ and $\pi'\in \Irr (\widetilde G')$, write
\[
  \mathrm m_\psi(\pi, \pi'):=\dim \Hom_{\widetilde G\times \widetilde G'}(\omega_\psi, \pi\boxtimes \pi')
\]
and define $\mathrm m_{\bar \psi}(\pi, \pi')$ similarly.
 We claim that
  \begin{equation}\label{dual00}
    \mathrm m_\psi(\pi, \pi')\neq 0\quad\textrm{if and only if} \quad \mathrm m_{\bar \psi}(\pi^\vee, {\pi'}^\vee)\neq 0.
 \end{equation}
 It is easy to see that \eqref{dimleq10} and \eqref{dual00} imply  Theorem \ref{T:howe3}.

To prove the claim, we first assume that  $\mathrm m_\psi(\pi, \pi')\neq 0$. Applying Lemma \ref{E:nonvan} to the pair $(G', G)$, we see that $\mathrm m_{\bar \psi}(\sigma, {\pi'}^\vee)\neq 0$ for some $\sigma\in \Irr(\widetilde G)$. The inequality \eqref{dimleq10}  then  implies that $\sigma\cong \pi^\vee$ and hence $\mathrm m_{\bar \psi}(\pi^\vee, {\pi'}^\vee)\neq 0$. Similarly, if $\mathrm m_{\bar \psi}(\pi^\vee, \sigma^\vee)\neq 0$ then $\mathrm m_\psi(\pi, \sigma)\neq 0$.  This proves the claim \eqref{dual00}, and therefore shows that Lemma \ref{E:nonvan} and Proposition \ref{T:mainq00}  imply  Theorem \ref{T:howe3}.
 \vskip 15pt

In view of the above, the main body of our paper will be devoted to the proof of Proposition \ref{T:mainq00}.
\vskip 10pt

\noindent {\bf Remarks}: (a) Reductive dual pairs as defined in this paper include the following case: $G$ is the quaternionic orthogonal group attached to a one-dimensional quaternionic skew Hermitian space, and $G'$ is the quaternionic symplectic group attached to a non-zero quaternionic Hermitian space (see the next section). In this case, $G'$ is strictly contained in the centralizer of $G$ in the symplectic group.

%(b) For type II  irreducible reductive dual pairs, Lemma \ref{E:nonvan} is a  consequence of the Godement-Jacquet zeta integral, which implies that $\theta_{\psi}(\pi)\neq 0$ for all $\pi\in \Irr(\widetilde G)$ when the size of $G$ is no larger than that of $G'$.

(b)
 Although in the statements of \cite[Theorems 1.3 and 1.4]{Ya} (see Lemma \ref{ry}) and \cite[Equalities (12)]{sz}, the authors assume that the base field $\rF$ has characteristic zero, their methods prove the same results  for all non-archimedean local field $\rF$ of characteristic not $2$.

(c)
For type II irreducible reductive dual pairs,  the identity (\ref{E:identity})  is a consequence of \cite[Theorem 1]{mi}, in which the explicit theta lifts are determined in terms of the Langlands parameters. For orthogonal-symplectic and unitary dual pairs,  (\ref{E:identity})  is a consequence of the MVW involution (\emph{cf.} \cite[Theorem 1.4]{S}).
 \vskip 15pt
\begin{center}{\bf Acknowledgements}\end{center}

 This paper is essentially completed during the conference in honor of Professor Roger Howe on the occasion of his 70th birthday. We thank the organisers of the conference
(James Cogdell, Ju-Lee Kim, Jian-Shu Li, David Manderscheid, Gregory Margulis, Cheng-Bo Zhu and Gregg Zuckerman) for their kind invitation to speak at the conference and for providing local support.
%\vskip 5pt

During the Howe conference, the first author presented his paper \cite{gt} with S. Takeda on the proof of the Howe duality conjecture for orthogonal-symplectic and unitary dual pairs and mentioned that the quaternionic case still needed to be addressed because of the lack of the MVW involution.  He expressed the hope that some trick could be found by the end of the conference to deal with the quaternionic case. The following day, the second author realised that a consequence of the conservation relation shown in his paper \cite{sz} with C.-B. Zhu could serve as a replacement for  the MVW-involution: this is the innocuous-looking statement  \eqref{E:nonvan0} above.  The two authors
were able to  verify the details in the next two days, thus completing the proof of the Howe duality conjecture in the quaternionic case.  It gives us great pleasure to dedicate this  paper to Roger Howe, who had initiated this whole area of research and formulated this conjecture at the beginning of his career. We hope that it gives him much satisfaction in seeing this conjecture completely resolved at the time of his retirement from Yale.

%\vskip 5pt

 W.T. Gan is partially supported by an MOE Tier Two grant R-146-000-175-112. B. Sun
is supported in part by the NSFC Grants 11222101 and 11321101.

\vskip 15pt

 \section{\bf The doubling method}\label{secdouble}
We will only treat the quaternionic case in the  proof of Proposition \ref{T:mainq00}, since it is previously known in all other cases.
Let $\rF$ be a nonarchimedean local field of characteristic not $2$, with $|\,\cdot\,|_\rF$ denoting the normalized absolute value on $\rF$.
 Let $\rD$ be a central division quaternion algebra over $\rF$, which is unique up to isomorphism. Denote by $\iota: \rD\rightarrow \rD$ the quaternion conjugation of $\rD$. We consider an $\epsilon$-Hermitian right $\rD$-vector space $U$, and an $\epsilon'$-Hermitian left $\rD$-vector space $V$, where $\epsilon=\pm 1$ and $\epsilon'=-\epsilon$. To be precise, $U$ is a finite dimensional right
$\mathrm D$-vector space, equipped with a non-degenerate $\mathrm
F$-bilinear map
\[
  \la\,,\,\ra_U : U\times U\rightarrow \mathrm D
\]
satisfying
\[
  \la u,u'\alpha \ra_U=\la u,u'\ra_U\, \alpha\quad\textrm{ and }\quad \la u,u'\ra_U=\epsilon \la u',u\ra_U^\iota, \qquad u,u'\in U, \,\alpha\in \mathrm D.
\]
Similarly, $V$ is a finite dimensional left $\rD$-vector space and is equipped with a form $\la \,,\,\ra_V:V\times V\rightarrow \rD$ with the analogous properties. The tensor product  $W:=U\otimes_\rD V$ is a symplectic space over $\rF$ under the bilinear form
 \begin{equation}
\label{tensor-form}
  \la u\otimes v, u'\otimes v'\ra_W:=\frac{\la u,u'\ra_U \,\la v,v'\ra_V^\iota+\la
  v,v'\ra_V\,\la u,u'\ra_U^\iota }{2},\quad u,u'\in U,\,v, v'\in V.
\end{equation}

Throughout the paper, we fix two quadratic (order at most $2$) characters $\chi_U, \chi_V :\rF^\times \rightarrow \{\pm 1\}$ determined by the discriminants of $U$ and $V$ respectively. More precisely, we have:
 \[
  \chi_V(\alpha)=\left( (-1)^{\dim V} \prod_{i=1}^{\dim V}\la e_i, e_i\ra_V \la e_i, e_i\ra_V^\iota, \alpha\right)_\rF,\qquad \alpha \in \rF^\times,
\]
where $e_1, e_2, \cdots, e_{\dim V}$ is an orthogonal basis of $V$, and $(\,,\,)_\rF$ denotes the quadratic Hilbert symbol for $\rF$.
Likewise,  one has the analogous definition for $\chi_U$. Note that if $\epsilon =1$, then the isometry class of $U$ is determined by its dimension, and $\chi_U$ only depends on the parity of $\dim U$; likewise, if $\epsilon' =1$, $\chi_V$ only depends on the parity of $\dim V$.
\vskip 5pt

Denote by $W^-$ the space $W$ equipped with the form scaled by $-1$. Write $W^\square:=W\oplus W^-$ for the orthogonal direct sum, which contains $W^\triangle:=\{(u,u)\in W^\square\mid u\in W\}$ as a Lagrangian subspace. Define  $U^-, V^-, U^\square, V^\square, U^\triangle, V^\triangle$ similarly.
Then we have obvious identifications of symplectic spaces
\[
  W^-=U^-\otimes_\rD V=U\otimes_\rD V^- \quad \textrm{and}\quad W^\square=U^\square\otimes_\rD V=U\otimes_\rD V^\square.
\]
Let  $\rG(U)$ denote the isometry group of $U$, and similarly for other groups.
Then we  have identifications
\[
  \rG(U)=\rG(U^-)\quad \textrm{ and }\quad \rG(V)=\rG(V^-),
\]
 and inclusions
\[
 \rG(U)\times \rG(U^-)\subset \rG(U^\square) \quad\textrm{ and }\quad \rG(V)\times \rG(V^-) \subset  \rG(V^\square).
\]
 Denote by  $\rP(U^\triangle)$ the parabolic subgroup of $\rG(U^\square)$ stabilizing $U^\triangle$. Likewise, denote by  $\rP(V^\triangle)$ the parabolic subgroup of $\rG(V^\square)$ stabilizing $V^\triangle$.

Let $\omega$ and $\omega^-$ be irreducible admissible smooth representations of $\rH(W)$ and $\rH(W^-)$, respectively, both with central character $\psi$.
Then the representation $\omega^\square:=\omega\boxtimes \omega^-$ of $\rH(W)\times \rH(W^-)$ descends to a representation of $\rH(W^\square)$ through the surjective homomorphism
\[
 \rH(W)\times\rH(W^-)\rightarrow \rH(W^\square),\quad ((u,\alpha), (v,\beta))\mapsto ((u,v), \alpha+\beta).
\]
This representation of $\rH(W^\square)$ uniquely extends to the group $\mathrm G(U^\square)\ltimes \mathrm H(W^\square)$ such that (\emph{cf.} \cite[Theorem 4.7]{sz})
\begin{equation}\label{lambdad}
    \lambda_\triangle (g\cdot \phi)=\chi_V({{\det}}(g|_{U^\triangle})) \, \abs{\det(g|_{U^\triangle})}_\rF^{\dim V}  \, \lambda_\triangle (\phi), \qquad \phi\in \omega^\square,
\ g\in \rP(U^\triangle),
\end{equation}
where $\lambda_\triangle$ denotes the unique (up to scalar multiplication) non-zero $W^\triangle$-invariant linear functional on $ \omega^\square$ and $\det$ denotes the reduced norm. Similarly, this representation of $\rH(W^\square)$ uniquely extends to the group $\mathrm G(V^\square)\ltimes \mathrm H(W^\square)$ such that
\[
    \lambda_\triangle (g'\cdot\phi)=\chi_U(\det(g'|_{V^\triangle})) \, \abs{\det(g'|_{V^\triangle})}_\rF^{\dim U}  \, \lambda_\triangle (\phi), \quad \phi\in \omega^\square,
\quad g'\in \rP(V^\triangle).
\]
We extend the representation $\omega$  to $(\rG(U)\times \rG(V))\ltimes \rH(W)$ and extend the representation $\omega^-$ to $(\rG(U^-)\times \rG(V^-))\ltimes \rH(W^-)$ such that
\[
  ((g,g')\cdot \phi)\otimes ((h,h')\cdot \phi^-)=gh\cdot(g'h'\cdot (\phi\otimes \phi^-))=g'h'\cdot (gh\cdot (\phi\otimes \phi^-)),
\]
 for all  $(g,h)\in \rG(U)\times \rG(U^-)$, $(g',h')\in \rG(V)\times \rG(V^-)$, $\phi\in\omega$ and $\phi^-\in \omega^-$. Then $\omega$ and $\omega^-$ are contragredient to each other with respect to the isomorphism
\[
  (\rG(U)\times \rG(V))\ltimes \rH(W)\rightarrow (\rG(U^-)\times \rG(V^-))\ltimes \rH(W^-), \quad ((g,g'),(u, \alpha))\mapsto ((g,g'),(u, -\alpha)).
\]
 If necessary, we also write $\omega_{U,V,\psi}$ for the representation $\omega$ of $(\rG(U)\times \rG(V))\ltimes \rH(W)$, and write $\omega^-_{U,V,\psi}$ for the representation $\omega^-$ of $(\rG(U^-)\times \rG(V^-))\ltimes \rH(W^-)$, to emphasize their dependence on $U,V$ and $\psi$.
 \vskip 5pt

 Thus, we have defined a splitting of (the pushout via $\{\pm 1\} \hookrightarrow \CC^\times$ of) the metaplectic cover $\widetilde{\rG}(U)$ and $\widetilde{\rG}(V)$ over $\rG(U)$ and $\rG(V)$ respectively, so that the Weil representation $\omega_{U,V,\psi}$ is a representation of the linear group $\rG(U) \times \rG(V)$. Such a splitting is unique over $\rG(U)$ if $U$ is quaternionic-Hermitian of dimension $>1$, but is not unique if $U$ is quaternionic-skew-Hermitian (as one can twist by quadratic characters of $\rG(U)$). For  the purpose of formulating and proving the Howe duality conjecture, there is no loss of generality in   working with a fixed splitting.
 \vskip 5pt

More precisely,  as in the introduction, for every  $\pi\in \Irr (\rG(U))$, put
\[
  \Theta_{\omega}(\pi):=(\omega\otimes \pi^\vee)_{\rG(U)},
\]
and define the theta lift $\theta_{\omega}(\pi)$ to be  the maximal semisimple quotient of $\Theta_{\omega}(\pi)$. Similarly, the theta lift $\theta_{\omega}(\pi')$ is defined for all $\pi'\in \Irr(\rG(V))$.
 The theta lifts with respect to other oscillator representations, such as $\theta_{\omega^-}$, are analogously defined.

Put
\[
  s_{U,V}:=\left(\dim U+\frac{\epsilon}{4}\right)-\left(\dim V+\frac{\epsilon'}{4}\right)\quad\textrm{and}\quad s_{V,U}:=\left(\dim V+\frac{\epsilon'}{4}\right)-\left(\dim U+\frac{\epsilon}{4}\right) = -s_{U,V}.
\]
The following is a reformulation of Proposition \ref{T:mainq00} in the quaternionic case, using the notations introduced above.

 \begin{prop} \label{T:mainq2}
 If $s_{U,V}>0$, then  for all $\pi,\sigma\in \Irr(\rG(U))$,
 \begin{equation}\label{dimleq1}
   \dim \Hom_{\rG(V)}(\theta_{\omega}(\pi)\otimes \theta_{\omega^-}(\sigma),\mathbb C)\leq \dim \Hom_{\rG(U)}(\pi\otimes \sigma, \mathbb C).
 \end{equation}
\end{prop}

The linear functional $\lambda_\triangle$ of \eqref{lambdad} induces a $\rG(U^\square)$-intertwining linear map
\begin{equation}\label{rallisq}
    \omega^\square\rightarrow \mathrm I(s_{V,U}), \quad \phi\mapsto (g\mapsto \lambda_\triangle(g\cdot \phi)).
\end{equation}
Here for each $s\in \mathbb C$,
 \[
 \mathrm I(s):=\mathrm{Ind}_{\mathrm P(U^\triangle)}^{\rG(U^\square)} \, \left(\chi_V\, |{\det}_{U^\triangle}|_\rF^s\right),
 \]
 where ${\det}_{U^\triangle}: \GL(U^\triangle)\rightarrow \rF^\times$ denotes the reduced norm map, and $\chi_V$ is viewed as a character of $\GL(U^\triangle)$ via the pullback through this map. Throughout this paper, $\mathrm{Ind}$ will denote the normalised parabolic induction functor.

Denote by $\rG(V)^\triangle$ the group $\rG(V)$ diagonally embedded in $\rG(V)\times \rG(V^-)$, to be viewed as a subgroup of $\rG(V^\square)$.

\begin{lem}\label{ry}
The linear map \eqref{rallisq} induces a $\rG(U^\square)$-intertwining linear embedding
\[
(\omega^\square)_{\rG(V)^\triangle}\hookrightarrow \mathrm I(s_{V,U}).
\]
If $s_{U,V}>0$, then there exists a surjective $\rG(U^\square)$-intertwining linear map
\[
\mathrm I(s_{U,V})\twoheadrightarrow(\omega^\square)_{\rG(V)^\triangle} \subset I(s_{V,U}).
\]
\end{lem}
\begin{proof}
The first assertion is due to Rallis, see \cite[Theorem II.1.1]{R} and \cite[Chapter 3, Theorem IV.7]{mvw}. The second one is proved in \cite[Theorems 1.3 and 1.4]{Ya}.
\end{proof}

Write $q_U$ for the Witt index of $U$. Fix two sequences
\begin{equation}\label{flagu0}
  0=X_0\subset X_1\subset \cdots \subset X_{q_U} \qquad \textrm{and}\qquad  X_{q_U}^*\supset \cdots \supset X_1^*\supset X_0^*=0
\end{equation}
of totally isotropic subspaces of $U$ such that for all $t=0,1,\cdots, q_U$,
\begin{equation}\label{flagu}
\left\{
  \begin{array}{ll}
    \dim X_t=\dim X_t^*=t;  \\
    X_t\cap X_t^*=0;  \ \textrm{ and}\\
    X_t\oplus X_t^*\textrm{ is non-degenerate}.
  \end{array}
\right.
\end{equation}
Denote by $U_t$ the orthogonal complement of $X_t\oplus X_t^*$ in $U$. Write $\mathrm P(X_t)$ and $\mathrm P(X_t^*)$ for the parabolic subgroups of $\rG(U)$ stabilizing $X_t$ and $X_t^*$, respectively. Then
\[
   \mathrm P(X_t)\cap \mathrm P(X_t^*)=\GL(X_t)\times \rG(U_t)
\]
is a common Levi factor of $\mathrm P(X_t)$ and $\mathrm P(X_t^*)$.

We need the following lemma (see \cite[Section 1]{kr05}).
\begin{lem} \label{L:key0}
Let $s\in \mathbb C$. As a representation of $\rG(U) \times
\rG(U^-)$, $\mathrm I(s)$
possesses an equivariant filtration
\[ 0=I_{-1}(s) \subset I_0(s) \subset  I_1(s) \subset\cdots\subset  I_{q_U}(s) =\mathrm I(s)  \]
with successive quotients
\[
   R_t(s) =I_t(s) /  I_{t-1}(s)   =     {\rm Ind}_{{\rm P}(X_t) \times {\rm P}(X_t)}^{\rG(U) \times
\rG(U^-)} \left( \left(\chi_V |{\det}_{X_{t}}|_\rF^{s + t} \boxtimes
\chi_V |{\det}_{X_t}|_\rF^{s+t} \right) \otimes   C^{\infty}_c(\rG(U_t))  \right), \]
where  $0\leq t\leq q_U$, and
\begin{itemize}
  \item $\det_{X_{t}}:\GL(X_t)\rightarrow \rF^\times$ denotes the reduced norm map, and $\chi_V$ is viewed as a character of $\GL(X_t)$ via the pullback through this map;
\item $\rG(U_t)\times \rG(U_t)$ acts on $C^{\infty}_c(
\rG(U_t))$ by  left-right translation.
\end{itemize}
In particular, ${R}_0(s) =  C^{\infty}_c(\rG(U))$ is the regular representation.
\end{lem}

In view of Lemma \ref{L:key0}, we make the following definition.
\vskip 5pt

\noindent{\bf \underline{Definition}:}
We say that an irreducible admissible smooth representation $\pi \boxtimes \sigma$ of $\rG(U) \times \rG(U^-)$ lies on the boundary of $\rm I(s)$ if
 \[
  \Hom_{\rG(U) \times \rG(U^-)}({R}_t(s),   \pi \boxtimes \sigma) \ne 0 \quad \text{ for some $0 < t  \leq q_U$,}
  \]
where ${R}_t(s)$ is as in Lemma \ref{L:key0}.

\vskip 5pt

Now we have:
\vskip 5pt

\begin{prop} \label{P:nonb}
Proposition \ref{T:mainq2} holds when  $\pi\boxtimes \sigma$ does not lie on the boundary of ${\rm I}(s_{U,V})$.
\end{prop}

\begin{proof}
 Consider the doubling see-saw
\[
\xymatrix{ \rG(U^\square)\ar@{-}[d]_{}\ar@{-}[dr]&\rG(V)\times
\rG(V^-)\ar@{-}[d]\\ \rG(U)\times \rG(U^-)\ar@{-}[ur]&\rG(V)^{\triangle} \,. }
\]
Given  $\pi,\sigma\in \Irr(\rG(U))$, the see-saw identity gives
\begin{equation}\label{E:see-saw}
\Hom_{\rG(U) \times \rG(U^-)} ( (\omega^\square)_{\rG(V)^\triangle}, \pi \boxtimes
\sigma) = \Hom_{\rG(V)}( \Theta_{\omega}(\pi) \otimes
\Theta_{\omega^-}(\sigma), \CC).
\end{equation}
Assume that $s_{U,V}>0$ and $\pi\boxtimes \sigma$ does not lie on the boundary of ${\rm I}(s_{U,V})$, then we have
\begin{eqnarray*}
  &&\Hom_{\rG(V)}( \theta_{\omega}(\pi) \otimes
\theta_{\omega^-}(\sigma), \CC) \\
  &\hookrightarrow& \Hom_{\rG(V)}( \Theta_{\omega}(\pi) \otimes
\Theta_{\omega^-}(\sigma), \CC)\\
   &=&\Hom_{\rG(U) \times \rG(U^-)} ( (\omega^\square)_{\rG(V)^\triangle}, \pi \boxtimes
\sigma) \\
   &\hookrightarrow&  \Hom_{\rG(U) \times \rG(U^-)} ( {\rm I}(s_{U,V}), \pi \boxtimes
\sigma) \qquad \qquad\textrm{(by Lemma \ref{ry})}\\
  &\hookrightarrow&  \Hom_{\rG(U) \times \rG(U^-)} (  C^{\infty}_c(\rG(U)), \pi \boxtimes
\sigma)  \qquad \qquad\textrm{(by Lemma \ref{L:key0})}\\
  &\cong&  \Hom_{\rG(U)} ( \pi \otimes
\sigma,\CC).
\end{eqnarray*}
This proves the proposition.
\end{proof}

\section{\bf Some induced representations}

To complete the proof of Proposition \ref{T:mainq2},  we need to consider representations $\pi \boxtimes \sigma$ of $\rG(U) \times \rG(U^-)$ which lie on the boundary of ${\rm I}(s_{U,V})$. To deal with these, we study in this section some parabolically induced representations which will play an important role later on.
\vskip 5pt

For  smooth representations $\rho$ of $\GL(X_t)$ and
$\sigma$ of $\rG(U_{t})$ ($0\leq t\leq q_U$), we write
\[  \rho \rtimes \sigma  :=  {\rm Ind}_{\mathrm P(X_t)}^{\rG(U)}  \rho \otimes \sigma. \]
  More generally, the parabolic subgroup $P$ of $\rG(U)$ stabilizing a flag
  \begin{equation}\label{flagx}
    0=X_{t_0}\subset X_{t_1}\subset X_{t_2}\subset\cdots\subset X_{t_a}
    \end{equation}
has a Levi factor of the form  $\GL(X_{t_1}) \times \GL(X_{t_2}/X_{t_1})\times \cdots \times \GL(X_{t_a}/X_{t_{a-1}})\times \rG(U_{t_a})$.  We set
\[
  \rho_1\times\cdots\times\rho_a\rtimes \sigma:=
\Ind_{P}^{\rG(U)}\rho_1\otimes\cdots\otimes\rho_a\otimes \sigma,
\]
where $\rho_i$ is a smooth representation of $\GL(X_{t_i}/X_{t_{i-1}})$ and $\sigma$ is a smooth
representation of $\rG(U_{t_a})$.
 Similarly,  for the general linear group $\GL(X_{t_a})$, we set
\[
\rho_1\times\cdots\times\rho_a:=
\Ind_Q^{\GL(X_{t_a})} \rho_1\otimes\cdots\otimes\rho_a,
\]
where $Q$ is the parabolic subgroup of $\GL(X_{t_a})$ stabilizing the flag \eqref{flagx}. Respectively write $\mathrm R_{X_t}$ and $\mathrm R_{X_t^*}$ for the normalized Jacquet functors attached to $\rP(X_t)$ and $\rP(X_t^*)$.
\vskip 5pt

Let $\eta:\rF^\times \rightarrow \CC^\times$ be a character of $\rF^\times$. Then $\eta^{\times a}$ ($0\leq a\leq q_U$) is an irreducible representation of $\GL(X_a)$ (\emph{cf.} \cite{Se}). Here $\eta$ is viewed as a character of  $\GL(X_t/X_{t-1})$  ($1\leq t\leq a$) via the pullback through the reduced norm map $\GL(X_t/X_{t-1})\rightarrow \rF^\times$. For every $\pi\in \Irr(\rG(U))$, define
\begin{equation}\label{meta}
  \mathrm m_{\eta}(\pi):=\max\{0\leq a\leq q_U\mid \pi \hookrightarrow \eta^{\times a}\rtimes \sigma \textrm{ for some }\sigma\in \Irr(\rG(U_a))\}.
   \end{equation}
   Here  $\pi \hookrightarrow \eta^{\times a}\rtimes \sigma$ means that there is an injective homomorphism from $\pi$ to $\eta^{\times a}\rtimes \sigma$ (similar notation will   be used without further explanation).

The rest of this section is devoted to a proof of the following proposition.

\begin{prop}\label{P:induced}
Assume that $\eta^2$ is non-trivial. Then for every $\pi\in \Irr(\rG(U))$, there is a unique representation $\pi_\eta\in
\Irr(\rG(U_a))$ such that $\pi\hookrightarrow \eta^{\times a}\rtimes \pi_\eta$, where  $a:=\mathrm m_{\eta}(\pi)$. Moreover,
\begin{equation}\label{piiso}
\left\{
  \begin{array}{ll}
    \pi_\eta\cong \left(\mathrm R_{X_a}(\pi)\otimes (\eta^{-1})^{\times a}\right)_{\GL(X_a)}\cong  \Hom_{\GL(X_a)}((\eta^{-1})^{\times a}, \mathrm R_{X_{a}^*}(\pi));\smallskip  \\
    \pi \textrm{ is isomorphic to the socle of } \eta^{\times a}\rtimes \pi_\eta; \smallskip \\
    \mathrm m_\eta(\pi)=\mathrm m_\eta(\pi^\vee);\smallskip \\
    (\pi^\vee)_\eta\cong (\pi_\eta)^\vee.
  \end{array}
\right.
\end{equation}

\end{prop}
\vskip 5pt

We begin with the following lemma.
\begin{lem}\label{L:ext}
For  all $\rho\in \Irr(\GL(X_a))$   ($0\leq a\leq q_U$) which is not isomorphic to $\eta^{\times a}$,
\[
   \mathrm{Ext}^i_{\GL(X_a)}(\eta^{\times a},\rho)=0\qquad (i\in \ZZ).
\]
\end{lem}
\begin{proof}
Since the Jacquet functor is exact and maps injective representations to injective representations, the second adjointness theorem of Bernstein implies that
\begin{equation}\label{extgl}
  \mathrm{Ext}^i_{\GL(X_a)}(\eta^{\times a},\rho)\cong \mathrm{Ext}^i_{(\rD^\times)^a}(\eta^{\boxtimes a}, \bar{\mathrm R} (\rho)).
\end{equation}
Here $\GL(X_a)$ is identified with $\GL_a(\rD)$ as usual, and $\bar{\mathrm R}$ denotes the normalized Jacquet functor attached to the minimal parabolic subgroup of $\GL_a(\rD)$ of  lower triangular matrices. Note that $\rho\ncong \eta^{\times a}$ implies that $\eta^{\boxtimes a}$ is not a subquotient of $\bar{\mathrm R}(\rho)$ (\emph{cf.} \cite[Chapter 3, Section 2.1, Theorem 18]{Be}). Then it is easy to see that the right hand side of \eqref{extgl} vanishes.

\end{proof}

By an easy homological algebra argument, Lemma \ref{L:ext} implies the following lemma.
\begin{lem}\label{extgg}
For  all $\rho\in \Irr(\GL(X_a))$  ($0\leq a\leq q_U$) which is not isomorphic to $\eta^{\times a}$, and all $\sigma, \sigma'\in \Irr(\rG(U_a))$,
\[
   \mathrm{Ext}^i_{\GL(X_a)\times \rG(U_a)}(\eta^{\times a}\boxtimes\sigma ,\rho\boxtimes \sigma')=0\qquad (i\in \ZZ).
\]
\end{lem}

From now on, we assume that the character $\eta^2 \ne 1$.
 \begin{lem}  \label{L:geometric0}
 Let $\sigma\in \Irr(\rG(U_a))$ ($0\leq a\leq q_U$). Assume that $\mathrm m_\eta(\sigma)=0$ (as defined in \eqref{meta}). Then
 \[
   \mathrm R_{X_a}(\eta^{\times a}\rtimes \sigma)\cong\left(\eta^{\times a}\boxtimes \sigma\right)\oplus \rho,
 \]
 where $\rho$ is a smooth representation of $\GL(X_a)\times \rG(U_a)$ which has no irreducible subquotient of the form $\eta^{\times a}\boxtimes \sigma'$ with $\sigma'\in \Irr(\rG(U_a))$. Consequently, the socle of $\eta^{\times a}\rtimes \sigma$ is irreducible.
 \end{lem}
 \begin{proof}
 Denote by $\rho$ the kernel of the natural surjective homomorphism
 \begin{equation}\label{sursplit}
    \mathrm R_{X_a}(\eta^{\times a}\rtimes \sigma)\twoheadrightarrow \eta^{\times a}\boxtimes \sigma.
 \end{equation}
As in the proof of \cite[Lemma 5.2]{gt}, using an explication of the Geometric Lemma of Bernstein-Zelevinsky (\emph{cf}. \cite[Lemma 5.1]{ta} and \cite{Ha}), the assumption of the lemma implies that $\rho$ contains no irreducible subquotient of the form $\eta^{\times a}\boxtimes \sigma'$ with $\sigma'\in \Irr(\rG(U_a))$. Then Lemma \ref{extgg} implies that the surjective homomorphism \eqref{sursplit} splits. This proves the first assertion of the lemma. The second assertion then easily follows as in \cite[Lemma 5.2]{gt}.
 \end{proof}

 In the rest of this section, let $\pi\in \Irr(\rG(U))$ and put $a:=\mathrm m_\eta(\pi)$. Then there is an irreducible representation $\sigma\in \Irr(\rG(U_a))$ such that $\pi \hookrightarrow\eta^{\times a}\rtimes \sigma$. Induction-by-steps shows that $\mathrm m_\eta(\sigma)=0$.

 \begin{lem}  \label{L:geometricirr2}
One has that
\[
   \mathrm R_{X_a}(\pi)\cong\left(\eta^{\times a}\boxtimes \sigma\right)\oplus \rho,
 \]
 where $\rho$ is a smooth representation of $\GL(X_a)\times \rG(U_a)$ which has no irreducible subquotient of the form $\eta^{\times a}\boxtimes \sigma'$ with $\sigma'\in \Irr(\rG(U_a))$.

\end{lem}

\begin{proof}
Since $\mathrm R_{X_a}(\pi)$ is a subrepresentation of $\mathrm R_{X_a}(\eta^{\times a}\rtimes \sigma)$ and has $\eta^{\times a}\boxtimes \sigma$ as an irreducible quotient, the lemma easily follows from Lemma \ref{L:geometric0}.
\end{proof}

 \begin{lem}  \label{L:geometricirr3}
One has that
\[
   \mathrm R_{X_a^*}(\pi)\cong\left((\eta^{-1})^{\times a}\boxtimes \sigma\right)\oplus \rho,
 \]
 where $\rho$ is a smooth representation of $\GL(X_a)\times \rG(U_a)$ which has no irreducible subquotient of the form $(\eta^{-1})^{\times a}\boxtimes \sigma'$ with $\sigma'\in \Irr(\rG(U_a))$.

\end{lem}
\begin{proof}
 Note that  $\rP(X_a)$ is conjugate to $\rP(X_a^*)$ by an element $w \in \rG(U)$ such that $w$ is the identity on $U_a$ and $w$ exchanges $X_a$ and $X_a^*$.
  Via conjugation by $w$, we see that Lemma \ref{L:geometricirr3} is equivalent to Lemma \ref{L:geometricirr2}.

\end{proof}

Lemma \ref{L:geometricirr2} and Lemma \ref{L:geometricirr3} imply that
\[
  \sigma\cong \left(\mathrm R_{X_a}(\pi)\otimes (\eta^{-1})^{\times a}\right)_{\GL(X_a)}\cong  \Hom_{\GL(X_a)}((\eta^{-1})^{\times a}, \mathrm R_{X_{a}^*}(\pi)).
\]
This proves the uniqueness assertion of Proposition \ref{P:induced}, as well as the first assertion of \eqref{piiso}. The  second assertion of \eqref{piiso} is then implied by the last assertion of Lemma \ref{L:geometric0}.

 \begin{lem}  \label{L:geometric4}
 One has that
 \[
  \pi^\vee \hookrightarrow\eta^{\times a}\rtimes \sigma^\vee.
 \]
 \end{lem}

 \begin{proof}
 Lemma \ref{L:geometricirr3} implies that
 \[
   (\eta^{-1})^{\times a}\boxtimes \sigma\hookrightarrow  \mathrm R_{X_a^*}(\pi).
 \]
 By dualizing and using the second adjointness theorem, we see that
 \[
   \mathrm R_{X_a}(\pi^\vee)\twoheadrightarrow \eta^{\times a}\boxtimes \sigma^\vee.
 \]
 This implies that $\pi^\vee \hookrightarrow\eta^{\times a}\rtimes \sigma^\vee$.

\end{proof}

Lemma \ref{L:geometric4} implies that $\mathrm m_\eta(\pi^\vee)\geq \mathrm m_\eta(\pi)$. The same argument shows that $\mathrm m_\eta(\pi)\geq \mathrm m_\eta(\pi^\vee)$. This proves that $\mathrm m_\eta(\pi^\vee)=\mathrm m_\eta(\pi)$. Lemma \ref{L:geometric4} then further implies that $(\pi^\vee)_\eta\cong (\pi_\eta)^\vee$. This finally finishes the proof of Proposition \ref{P:induced}.

\vskip 5pt

 \section{\bf Induced representations and theta correspondence}

In this section, we apply the results of the previous section to the theta correspondence.
Write $\eta': \rF^\times\rightarrow \CC^\times$ for the character such that
\[
  \eta' \cdot \chi_V=\eta \cdot \chi_U.
\]
Then $(\eta')^2 \ne 1$  since   $\eta^2 \ne 1$.
Denote  by $q_V$ the Witt index of $V$. Similarly to \eqref{flagu0}, we fix two sequences
\[
  0=Y_0\subset Y_1\subset \cdots \subset Y_{q_V}\qquad \textrm{and}\qquad Y_{q_V}^*\supset \cdots \supset Y_1^*\supset Y_0^*=0
\]
of totally isotropic subspaces of $V$ with the analogous properties as in \eqref{flagu0}. We apply the analogous notation as in the last section  to the space $V$. In particular, $\mathrm m_{\eta'}(\pi')$ is defined for every $\pi'\in \Irr(\rG(V))$. Define $\pi_\eta$ ($\pi\in \Irr(\rG(U))$ and $\pi'_{\eta'}$ as in Proposition \ref{P:induced}.

For all integers $0\leq a\leq q_U$ and $0\leq k\leq q_V$, write $\omega_{a,k}:=\omega_{U_a, V_k,\psi}$, which is an irreducible smooth representation of $(\rG(U_a)\times \rG(V_k))\ltimes \rH(U_a\otimes_\rD V_k)$, as defined in Section \ref{secdouble}.
\vskip 5pt

The rest of this section is devoted to a proof of the following key proposition.
\begin{prop}\label{induction}
Assume that
\[
  \eta\neq \chi_V \abs{\,\cdot\,}_\rF^{s_{V,U}+1} \qquad \textrm{and}\qquad   \eta'\neq \chi_U \abs{\cdot}_\rF^{s_{U,V}+1}.
\]
Then for all $\pi\in \Irr(\rG(U))$ and $\pi'\in \Irr(\rG(V))$ such that $\Hom_{\rG(U)\times \rG(V)}(\omega, \pi\boxtimes \pi')\neq 0$, one has
\[
  \mathrm m_{\eta}(\pi)=\mathrm m_{\eta'}(\pi'),
\]
and there is a linear embedding
\[
   \Hom_{\rG(U)\times \rG(V)}(\omega, \pi\boxtimes \pi')\hookrightarrow \Hom_{\rG(U_a)\times \rG(V_a)}(\omega_{a,a}, \pi_\eta\boxtimes \pi'_{\eta'}),
\]
where $a:=\mathrm m_{\eta}(\pi)$.
\end{prop}

For each right $\rD$-vector space $X$, write $X^\iota$ for the left $\rD$-vector which equals $X$ as an abelian group and whose scalar multiplication is given by
\[
  \alpha v:=v \alpha^\iota,\qquad \alpha\in \rD, \,v\in X^\iota.
\]
We first recall the well-known computation of the Jacquet module of the Weil representation (see \cite[Theorem 2.8]{k83} and \cite[Chapter 3, Section IV.5]{mvw}).
 \vskip 5pt

 \begin{lem} \label{L:kudla}
For each $0\leq a\leq q_U$, the normalized Jacquet module $\mathrm R_{X_a}(\omega)$ has
a $\GL(X_a)\times \rG(U_a)\times \rG(V)$-equivariant filtration
\[
  \mathrm R_{X_a}(\omega) = R_0 \supset R_1 \supset \cdots \supset R_{a'}
\supset R_{a'+1} = 0 \]
whose successive quotient
is
\[
 J_k:=R_k/R_{k+1}\cong \Ind^{\GL(X_a) \times \rG(U_a) \times \rG(V)}_{\rP(X_{a-k}, X_a)
\times \rG(U_a) \times \rP(Y_k)}  \left(\chi_V
|{\det}_{X_{a-k}}|_\rF^{s_{V,U}+a-k} \otimes
C^{\infty}_c(\Isom(X_a^\iota/X_{a-k}^\iota, Y_k)) \otimes
\omega_{a,k}\right),
\]
 where
\begin{itemize}
\item $a':=\min\{a, q_V\}$ and $0\leq k\leq a'$;
\item $\rP(X_{a-k}, X_a)$ is the parabolic subgroup of $\GL(X_a)$ stabilizing $X_{a-k}$;
\item $\det_{X_{a-k}}: \GL(X_{a-k})\rightarrow \rF^\times$ denotes the reduced norm map, and $\chi_V$ is viewed as a character of $\GL(X_{a-k})$ via the pullback through this map;
\item $\Isom(X_a^\iota/X_{a-k}^\iota, Y_k)$ is the set of $\rD$-linear isomorphisms from $X_a^\iota/X_{a-k}^\iota$ to $Y_k$, and $\GL(X_a/X_{a-k})\times\GL(Y_k)$ acts on
  $C_c^\infty(\Isom(X_a^\iota/X_{a-k}^\iota,Y_k))$ as
  \[
   ((b,c)\cdot f)(g)=\chi_V (\det b)\chi_U(\det
  c)f(c^{-1}g b),
  \]
   for $(b,c)\in \GL(X_a/X_{a-k})\times\GL(Y_k)$, $f\in
  C_c^\infty(\Isom(X_a^\iota/X_{a-k}^\iota,Y_k))$ and $g\in\Isom(X_a^\iota/X_{a-k}^\iota,Y_k)$.
\end{itemize}
In particular, if $a'=a$, then the bottom piece
of the filtration is
\[ J_a \cong \Ind^{\GL(X_a) \times G(U_a) \times \rG(V)}_{\GL(X_a)
\times \rG(U_a) \times \rP(Y_a)} \left(C^{\infty}_c(\Isom(X_a^\iota,Y_a)) \otimes \omega_{a,a}\right). \]
\end{lem}

The following lemma is an observation of \cite{gt}.
\begin{lem}\label{boundo}
Let $a$, $k$ and $J_k$ be as in Lemma \ref{L:kudla}. Assume that $\eta\neq \chi_V \abs{\,\cdot\,}_\rF^{s_{V,U}+1}$. Then for all $\sigma\in \Irr(\rG(U_a))$ and $\pi'\in \Irr(\rG(V))$,
\begin{equation}\label{homggg}
  \Hom_{\GL(X_a) \times \rG(U_a) \times \rG(V)}(J_k, \eta^{\times a}\boxtimes \sigma\boxtimes \pi')=0
\end{equation}
whenever $k\neq a$.
\end{lem}
\begin{proof}
Using the second adjointness theorem, it suffices to show that
\[
   \Hom_{\GL(X_{a-k})}(\chi_V \abs{{\det}_{X_{a-k}}}_\rF^{s_{V,U}+a-k},  \bar{\mathrm R}_{X_{a-k},X_a}(\eta^{\times a}))=0,
\]
where $\bar{\mathrm R}_{X_{a-k},X_a}$ denotes the normalized Jacquet functor attached  to the parabolic subgroup of $\GL(X_a)$ stabilizing a complement of $X_{a-k}$ in $X_a$.
By analysing the cuspidal data, we know that every irreducible subrepresentation of $\bar{\mathrm R}_{X_{a-k},X_a}(\eta^{\times a})$ is isomorphic to $\eta^{\times (a-k)}\boxtimes \eta^{\times k}$, as a representation of $\GL(X_{a-k})\times \GL(X_a/X_{a-k})$. Therefore the lemma follows.
\end{proof}

Now we come to the proof of Proposition \ref{induction}. Put $a:=\mathrm m_{\eta}(\pi)$. Then we have
\begin{eqnarray*}
  0 &\neq &\Hom_{\rG(U)\times \rG(V)}(\omega, \pi\boxtimes \pi')\\
  &\hookrightarrow& \Hom_{\rG(U)\times \rG(V)}(\omega, (\eta^{\times a}\rtimes \pi_\eta)\boxtimes \pi')\\
   &=&\Hom_{\GL(X_a)\times \rG(U_a) \times \rG(V)} ( \mathrm R_{X_a}(\omega), \eta^{\times a}\boxtimes \pi_\eta \boxtimes
\pi') \\
   &\hookrightarrow&  \Hom_{\GL(X_a)\times \rG(U_a) \times \rG(V)} (J_a, \eta^{\times a}\boxtimes \pi_\eta \boxtimes
\pi')  \qquad \qquad\textrm{(by Lemma \ref{boundo})}\\
  &\cong &  \Hom_{\GL(X_a)\times \rG(U_a) \times \GL(Y_a)\times \rG(V_a)} (C^{\infty}_c(\Isom(X_a^\iota,Y_a)) \otimes \omega_{a,a}, \eta^{\times a}\boxtimes \pi_\eta \boxtimes
\mathrm R_{Y_a^*}(\pi')) \\
 && \quad \qquad\textrm{(by the second ajointness theorem)}\\
  &\cong &  \Hom_{\rG(U_a) \times \GL(Y_a)\times \rG(V_a)} (({\eta'}^{-1})^{\times a}\boxtimes \omega_{a,a}, \pi_\eta \boxtimes
\mathrm R_{Y_a^*}(\pi'))\\
 & \cong &  \Hom_{\rG(U_a) \times \rG(V_a)} (\omega_{a,a}, \pi_\eta \boxtimes \pi'_a),\\
\end{eqnarray*}
where
\[
 \pi_a':=\Hom_{\GL(Y_a)}(({\eta'}^{-1})^{\times a}, \mathrm R_{Y_a^*}(\pi')).
\]
Therefore $\pi_a'\neq 0$, and hence
\[
  \Hom_{\GL(Y_a)\times \rG(V_a)}(({\eta'}^{-1})^{\times a}\boxtimes {\pi_a'}, \mathrm R_{Y_a^*}(\pi'))\neq 0.
\]
Dualizing and using the second adjointness theorem, we see that
\[
  \Hom_{\GL(Y_a)\times \rG(V_a)}(\mathrm R_{Y_a}({\pi'}^\vee), {\eta'}^{\times a}\boxtimes {\pi_a'}^\vee)\neq 0.
\]
This proves that
\[
\mathrm m_{\eta'}(\pi')=\mathrm m_{\eta'}({\pi'}^\vee)\geq a=\mathrm m_{\eta}(\pi).
\]
The same argument shows that $\mathrm m_{\eta}(\pi)\geq \mathrm m_{\eta'}(\pi')$, and hence  $\mathrm m_{\eta'}(\pi')=\mathrm m_{\eta}(\pi)$. Therefore  $\pi'_a\cong\pi'_{\eta'}$ by Proposition \ref{P:induced}. This finishes the proof of Proposition \ref{induction}.

% \section{\bf Degenerate principal series}

 \section{\bf Proof of Proposition \ref{T:mainq2}}

In this section, we finish the proof of Proposition \ref{T:mainq2} by induction on $\dim U$.  As in Proposition \ref{T:mainq2}, let $\pi,\sigma\in \Irr(\rG(U))$ and assume that $s_{U,V}>0$.  In view of Proposition \ref{P:nonb}, we may assume that  $\pi\boxtimes \sigma$ lies on the boundary of ${\rm I}(s_{U,V})$. Then there is an integer $0 < t  \leq q_U$ such that
  \begin{equation}\label{bound2}
  \Hom_{\rG(U) \times \rG(U^-)}({R}_t(s_{U,V}),   \pi \boxtimes \sigma) \ne 0.
  \end{equation}
Note that
\begin{eqnarray}\label{countdim}
   &&  \dim  \Hom_{\rG(V)}(\theta_{\omega}(\pi)\otimes \theta_{\omega^-}(\sigma),\mathbb C) \\
  \nonumber &=& \sum_{\pi'\in \Irr(\rG(V))}  \dim \Hom_{\rG(U)\times \rG(V)}(\omega, \pi\boxtimes \pi')\, \cdot\,
    \dim \Hom_{\rG(U)\times \rG(V)}(\omega^-, \sigma\boxtimes {\pi'}^\vee).
\end{eqnarray}
We assume that the value of the above equality is non-zero, as Proposition \ref{T:mainq2} is otherwise trivial.
Then there is an irreducible representation $\pi'\in \Irr(\rG(V))$ such that
\begin{equation}\label{homdouble}
  \Hom_{\rG(U)\times \rG(V)}(\omega, \pi\boxtimes \pi')\neq 0\quad \textrm{and}\quad \Hom_{\rG(U)\times \rG(V)}(\omega^-, \sigma\boxtimes {\pi'}^\vee)\neq 0.
\end{equation}

By the second adjointness theorem, \eqref{bound2} implies that
\begin{equation}\label{homgt}
  \Hom_{\GL(X_t)} (\chi |{\det}_{X_{t}}|_\rF^{s_{U,V} + t}, \mathrm R_{X_t^*} (\pi))\neq 0.
\end{equation}
Put
\[
  \eta:=\chi_V \,\abs{\,\cdot\,}^{s_{V,U}-2t+1}\quad \textrm{and}\quad \eta':=\chi_U \,\abs{\,\cdot\,}^{s_{V,U}-2t+1}.
\]
Using the second adjointness theorem and the Langlands parameter of the character $\chi_V |{\det}_{X_{t}}|_\rF^{s_{V,U}-t}$, \eqref{homgt} implies that
\[
  \mathrm m_{\eta}(\pi)=\mathrm m_{\eta}(\pi^\vee)>0.
\]
Noting that
\[
\eta\neq \chi_V \abs{\,\cdot\,}_\rF^{s_{V,U}+1} \qquad \textrm{and}\qquad   \eta'\neq \chi_U \abs{\cdot}_\rF^{s_{U,V}+1},
\]
Proposition \ref{induction} (and its analog for $\omega^-$) then implies that
\[
   \mathrm m_{\eta}(\pi)=\mathrm m_{\eta'}(\pi')=\mathrm m_{\eta'}({\pi'}^\vee)=\mathrm m_{\eta}(\sigma).
\]

By  the induction assumption, Proposition \ref{T:mainq2} holds for the pair $(U_a, V_a)$, where $a:=\mathrm m_{\eta}(\pi)$. As we have seen at the end of the introduction,  this  implies that Theorem \ref{T:howe3}  holds for the pair $(\rG(U_a), \rG(V_a))$. Together with Proposition \ref{induction}, this implies that
\begin{equation}\label{sigmaa0}
  {\pi'}_{\eta'}\cong\theta_{\omega_{a,a}}(\pi_\eta)
\end{equation}
and
\begin{equation}\label{sigmaa}
   \pi_\eta\cong\theta_{\omega_{a,a}}(\pi'_{\eta'})\qquad  \textrm{and}\qquad \sigma_\eta\cong \theta_{\omega^-_{a,a}}((\pi'_{\eta'})^\vee).
\end{equation}
Here $\omega^-_{a,a}:=\omega^-_{U_a,V_a,\psi}$. Proposition \ref{induction} and \eqref{sigmaa0} imply that $\pi'$ is isomorphic to the socle of ${\eta'}^{\times a}\rtimes \theta_{\omega_{a,a}}(\pi_\eta)$. Therefore,  there is a unique $\pi'\in \Irr(\rG(U))$ which satisfies \eqref{homdouble}. Then
Proposition \ref{induction} implies that the value of \eqref{countdim} is  $1$.

On the other hand, \eqref{sigmaa} and  the induction assumption imply that $\pi_\eta^\vee\cong \sigma_\eta$, which further implies that $\pi^\vee\cong \sigma$ by
Proposition \ref{P:induced}. Therefore \eqref{dimleq1} of Proposition \ref{T:mainq2} is an equality. This finishes the proof of Proposition \ref{T:mainq2}.

\vskip 5pt

%\noindent {\bf Remark}:

\end{document}